\documentclass{siamltex}

\usepackage[latin9]{inputenc}
\usepackage{mathrsfs}
\usepackage{amsmath}
\usepackage{amssymb}
\usepackage{esint}
\usepackage{enumerate}
\usepackage{bbm}
\usepackage{stmaryrd} 
\usepackage{graphicx}

\usepackage[usenames]{color}

\newtheorem{remark}[theorem]{Remark}

\newcommand{\N}{\mathbb{N}}
\newcommand{\R}{\mathbb{R}}

\newcommand{\Ld}{L_\diamond}
\newcommand{\Hd}{H_\diamond}
\newcommand{\LL}{\mathcal{L}}

\newcommand{\sS}{\mathbb{S}}
\newcommand{\1}{\ensuremath{\mathbbm{1}}}

\newcommand{\dx}[1][x]{\ensuremath{\,{\rm{d}} #1}}

\def\norm#1{\hspace{0.2ex} \|#1\| \hspace{0.2ex}} 
 
\newcommand{\kommentar}[1]{}

\begin{document}

\title{The Calder\'on problem with finitely many unknowns is equivalent to convex semidefinite optimization}

\author{Bastian Harrach\footnotemark[2]}
\renewcommand{\thefootnote}{\fnsymbol{footnote}}


\footnotetext[2]{Institute for Mathematics, Goethe-University Frankfurt, Frankfurt am Main, 
Germany (harrach@math.uni-frankfurt.de)}

\maketitle
\begin{abstract}
We consider the inverse boundary value problem of determining a coefficient function in an elliptic partial differential equation from knowledge of the associated Neumann-Dirichlet-operator. The unknown coefficient function is assumed to be piecewise constant with respect to a given pixel partition, and upper and lower bounds are assumed to be known a-priori.

We will show that this Calder\'on problem with finitely many unknowns can be equivalently formulated as a minimization problem for a linear cost functional with a convex non-linear semidefinite constraint. We also prove error estimates for noisy data, and extend the result to the practically relevant case of finitely many measurements, where the coefficient is to be reconstructed from a finite-dimensional Galerkin projection of the Neumann-Dirichlet-operator.

Our result is based on previous works on Loewner monotonicity and convexity of the Neumann-Dirichlet-operator, and the technique of localized potentials. It connects the emerging fields of inverse coefficient problems and semidefinite optimization.
\end{abstract}

\begin{keywords}
Inverse coefficient problem, Calder\'on problem, finite resolution, semidefinite optimization, Loewner monotonicity and convexity
\end{keywords}

\begin{AMS}
35R30, 
90C22 
\end{AMS}

\section{Introduction}

We consider the Calder\'on problem of determining the spatially dependent coefficient function $\sigma$ in the elliptic partial differential equation
\begin{equation*}
\nabla\cdot (\sigma \nabla u)=0
\end{equation*}
from knowledge of the associated (partial data) Neumann-Dirichlet-operator $\Lambda(\sigma)$, cf.\ section \ref{subsect:setting_Calderon} for the precise mathematical setting.
The coefficient function $\sigma$ is assumed to be piecewise constant with respect to a given pixel partition of the underlying imaging domain, so that only finitely many unknowns have to be reconstructed. We also assume that upper and lower bounds $b>a>0$ are known a-priori, so that the pixel-wise values of the unknown coefficient function can be identified with a vector in $[a,b]^n\subset \mathbb{R}^n$, where $n\in \N$ is the number of pixels. 

In this paper, we prove that the problem can be equivalently reformulated as a convex optimization problem where a linear cost function is to be minimized under a non-linear convex semi-definiteness constraint. Given $\hat Y=\Lambda(\hat \sigma)$, the
vector of pixel-wise values of $\hat\sigma$ is shown to be the unique minimizer of 
\[
c^T \sigma\to \text{min!}\quad \text{s.t.} \quad \sigma\in [a,b]^n,\ \Lambda(\sigma)\preceq\hat Y,
\]
where $c\in \mathbb{R}^n$ only depends on the pixel partition, and on the upper and lower bounds $a,b>0$. The symbol ``$\preceq$'' denotes the semidefinite (or Loewner) order, and $\Lambda$ is shown to be convex with respect to this order, so that the admissable set of this optimization problem is convex.

We also prove an error estimate for the case of noisy measurements $Y^\delta\approx \Lambda(\hat \sigma)$, and show that our results still hold for the case of finitely many (but sufficiently many) measurements. 

Let us give some more remarks on the origins and relevance of this result.
The Calder\'on problem \cite{calderon1980inverse,calderon2006inverse} has received immense attention in the last decades due to its relevance for non-destructive testing and medical imaging applications, and its theoretical importance in studying inverse coeffient problems. We refer to \cite{krupchyk2016calderon,caro2016global} for recent theoretical breakthroughs, and the books \cite{mueller2012linear,seo2012nonlinear,adler2021electrical} for the 
prominent application of electrical impedance tomography.

In practical applications, only a finite number of measurements can be taken and the unknown coefficient function can only be reconstructed up to a certain resolution. For the resulting finite-dimensional non-linear inverse problems, uniqueness results have been obtained only recently in \cite{alberti2019calderon,harrach2019uniqueness}.

Numerical reconstruction algorithms for the Calder\'on problem and related inverse coefficient problems are typically based on Newton-type iterations or on minimizing a non-convex regularized data fitting functional. Both approaches highly suffer from the problem of local convergence, resp., local minima, and therefore require a good initial guess close to the unknown solution which is usually not available in practice. We refer to the above mentioned books \cite{mueller2012linear,seo2012nonlinear,adler2021electrical} for an overview on this topic, and point out the result in \cite{Lec08} that shows a local convergence result for the Newton method for EIT with finitely many measurements and unknowns. For a specific infinite-dimensional setting, a convexification idea was developed in \cite{klibanov2019convexification}.
Moreover, knowing $\hat Y=\Lambda(\hat\sigma)$ (i.e., infinitely many measurements) and the pixel partition, the values in $\hat\sigma$ can also be recovered one by one with globally convergent one-dimensional monotonicity tests,
and these tests can be implemented as in \cite{garde2020reconstruction,garde2022simplified} without knowing the upper and lower conductivity bounds. But, to the knowledge of the author, these ideas do not carry over to the case of finitely many measurements, and, for this practically important case, the problem of local convergence, resp., local minima, remained unsolved. 

The new equivalent convex reformulation of the Calder\'on problem with finitely many unknowns presented in this work 
connects the emerging fields of inverse problems in PDEs and semidefinite optimization. It is based on previous works on Loewner monotonicity and convexity, and the technique of localized potentials \cite{gebauer2008localized,harrach2013monotonicity}, and extends the recent work \cite{harrach2021solving} to the Calder\'on problem. The origins of these ideas go back to inclusion detection algorithms
such as the factorization and monotonicity method \cite{kirsch1998characterization,bruhl2000numerical,tamburrino2002new},
and the idea of overcoming non-linearity in such problems \cite{harrach2010exact}.

The structure of this article is as follows. In section 2, we demonstrate on a simple, yet illustrative example how non-linear inverse coefficient problems such as the Calder\'on problem suffer from local minima. In section 3 we then formulate our main results on reformulating the Calder\'on problem as a 
convex minimization problem. In section 4 the results are proven, and section 5 contains some conclusions and an outlook.

\section{Motivation: The problem of local minima}\label{section:Motivation}

To motivate the importance of finding convex reformulations, we first give a simple example on how drastically inverse coefficent problems such as the Calder\'on problem can suffer from local minima. We consider
\begin{equation}\label{eq:Motivation_EIT}
\nabla\cdot (\sigma \nabla u)=0
\end{equation}
in the two-dimensional unit ball $\Omega=B_1(0)\subset \mathbb{R}^2$. The coefficient $\sigma$ is assumed to be the radially-symmetric piecewise-constant function
\[
\sigma(x)=\left\{ \begin{array}{l c r@{\,} l} 
1 & \text{ for } &1&>\norm{x}\geq r_1,\\
\sigma_1 &  \text{ for } &r_1&> \norm{x}\geq r_2,\\
\sigma_2 &  \text{ for } &r_2&> \norm{x},
\end{array}\right.
\]
with known radii $1>r_1>r_2>0$, but unknown values $\sigma_1,\sigma_2>0$, cf.\ figure~\ref{fig:Setting}
for a sketch of the setting with $r_1:=0.5$ and $r_2:=0.25$.

\begin{figure}
\begin{center}
\mbox{\includegraphics[height=0.25\textheight]{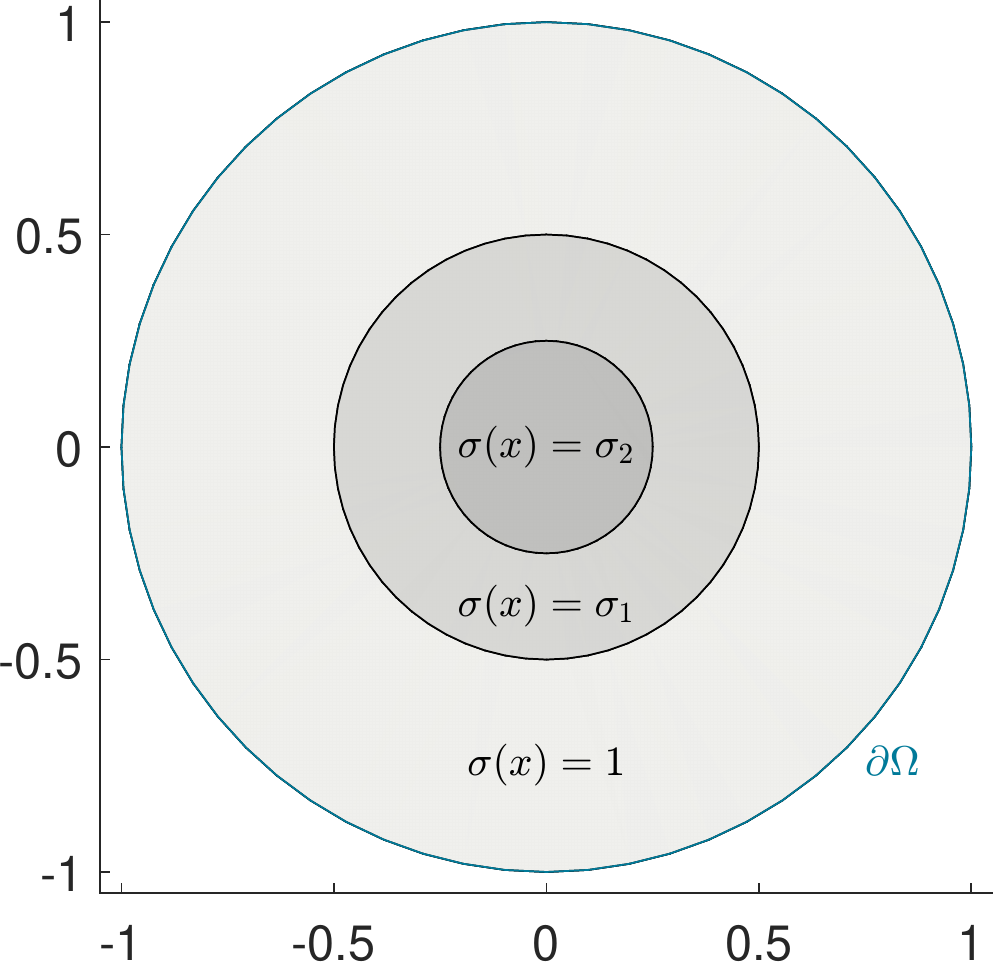}} 
\end{center}
\caption{A simple example with two unknown values.}
\label{fig:Setting}
\end{figure}

We aim to reconstruct the two unknown values $\sigma=(\sigma_1,\sigma_2)$ from the Neumann-Dirichlet-operator (NtD)
\[
\Lambda(\sigma):\ g\mapsto u\vert_{\partial \Omega}, \quad
\text{where $u$ solves \eqref{eq:Motivation_EIT} with $\sigma\partial_\nu u\vert_{\partial \Omega}=g$.}
\]

For this simple geometry, the NtD can be calculated analytically. 
Using polar coordinates $(r,\varphi)$, it is easily checked that, for each $j\in \N$, the function
\[
u(r,\varphi):=\left\{ \begin{array}{r c r@{\,} l} 
r^j \sin(j\varphi) & \text{ for } & r_2&>r,\\[+.25ex]
\frac{1}{2}(a_j r^j + b_j r^{-j})\sin(j\varphi) & \text{ for } & r_1& >r\geq r_2,\\[+.5ex]
\frac{1}{4}(c_j r^j + d_j r^{-j})\sin(j\varphi) & \text{ for } & 1&>r\geq r_1.
\end{array}\right.
\]
solves \eqref{eq:Motivation_EIT} (in the weak sense) if $a_j,b_j,c_j,d_j\in \mathbb{R}$ fulfill the following four interface conditions 
\begin{alignat*}{2}
u\vert_{\partial B_{r_2}^-}&=u\vert_{\partial B_{r_2}^+}, & \qquad 
u\vert_{\partial B_{r_1}^-}&=u\vert_{\partial B_{r_1}^+},\\
\sigma\partial_\nu u\vert_{\partial B_{r_2}^-}&=\sigma\partial_\nu u\vert_{\partial B_{r_2}^+}, & \qquad 
\sigma\partial_\nu u\vert_{\partial B_{r_1}^-}&=\sigma\partial_\nu u\vert_{\partial B_{r_1}^+}.
\end{alignat*}
%
This is equivalent to $a_j,b_j,c_j,d_j\in \mathbb{R}$ solving the linear system
\begin{alignat*}{2}
1&=\frac{1}{2}(a_j + b_j r_2^{-2j}), & \qquad 
a_j  + b_j r_1^{-2j}&=\frac{1}{2}(c_j  + d_j r_1^{-2j}),\\
\frac{\sigma_2}{\sigma_1} &= \frac{1}{2}(a_j - b_j r_2^{-2j}), & \qquad
\sigma_1(a_j   - b_j  r_1^{-2j})&=\frac{1}{2}(c_j  - d_j  r_1^{-2j}),
\end{alignat*}
and from this we easily obtain that
\begin{align*}
a_j&=1+\frac{\sigma_2}{\sigma_1},\\
b_j&=\left( 1-\frac{\sigma_2}{\sigma_1} \right) r_2^{2j},\\
c_j&=a_j  + b_j r_1^{-2j} + \sigma_1(a_j   - b_j  r_1^{-2j})\\
&=\left(\frac{1}{\sigma_1}+1\right)(\sigma_1+\sigma_2)+\left(\frac{1}{\sigma_1}-1\right)(\sigma_1-\sigma_2)\frac{r_2^{2j}}{r_1^{2j}},\\
d_j   &=  a_j r_1^{2j} + b_j  - \sigma_1(a_j r_1^{2j}  - b_j )\\
&=  \left(\frac{1}{\sigma_1}-1\right)\left(\sigma_1+\sigma_2\right)r_1^{2j}  + \left(\frac{1}{\sigma_1}+1\right)\left( \sigma_1-\sigma_2 \right) r_2^{2j}.
\end{align*}

The Dirichlet- and Neumann boundary values of $u$ are 
\begin{align*}
u\vert_{\partial B_1(0)}&=\frac{1}{4}(c_j+d_j)\sin(j\varphi),\quad \text{ and } \quad
\sigma\partial_\nu u\vert_{\partial B_1(0)}=\frac{j}{4}(c_j-d_j)\sin(j\varphi),
\end{align*}
so that 
\[
\Lambda(\sigma)\sin(j\varphi)=\lambda_j \sin(j\varphi) \quad \text{ with } \quad
\lambda_j:=\frac{c_j+d_j}{j(c_j-d_j)}.
\]
By rotational symmetry, the same holds with $\sin(\cdot)$ replaced by $\cos(\cdot)$.
Hence, with respect to the standard $L^2$-orthonormal basis of trigonometric functions
\[
\left\{g_1,g_2,g_3,\ldots\right\}=
\left\{ \frac{1}{\pi}\sin(\varphi),\frac{1}{\pi}\cos(\varphi),\frac{1}{\pi}\sin(2\varphi),\ \ldots\right\}\subseteq \Ld^2(\partial \Omega),
\]
the Neumann-Dirichlet-operator $\Lambda(\sigma)\in \mathcal{L}(\Ld^2(\partial \Omega))$ can be written as the infinite-dimensional diagonal matrix
\[
\Lambda(\sigma)=\begin{pmatrix}\lambda_1\\ & \lambda_1\\ & & \lambda_2\\ & & & \ddots \end{pmatrix},
\]
with $\lambda_j$ depending on $\sigma=(\sigma_1,\sigma_2)\in \mathbb{R}^2$ as given above,
and $\Ld^2(\partial \Omega)$ denoting the space of $L^2$-functions with vanishing integral mean on $\partial \Omega$.

We assume that we can only take finitely many measurements of $\Lambda(\sigma)$. For this example, we choose the measurements to be the upper left $6\times 6$-part of this matrix, i.e., the Galerkin projektion of $\Lambda(\sigma)$ to $\operatorname{span}\{g_1,\ldots,g_6\}$,
\[
F(\sigma):=\left( \int_{\partial\Omega} g_j \Lambda(\sigma)g_k \dx[s] \right)_{j,k=1,\ldots,6}\in \mathbb{R}^{6\times 6},
\]
and try to reconstruct $\sigma\in \R^2$ from $F(\sigma)\in \R^{6\times 6}$. Note that, effectively, we thus aim to reconstruct two unknowns $\sigma=(\sigma_1,\sigma_2)$ from three independent measurements $\lambda_1$, $\lambda_2$, and $\lambda_3$.

\begin{figure}
\begin{center}
\begin{tabular}{c c}
\mbox{\includegraphics[height=0.25\textheight]{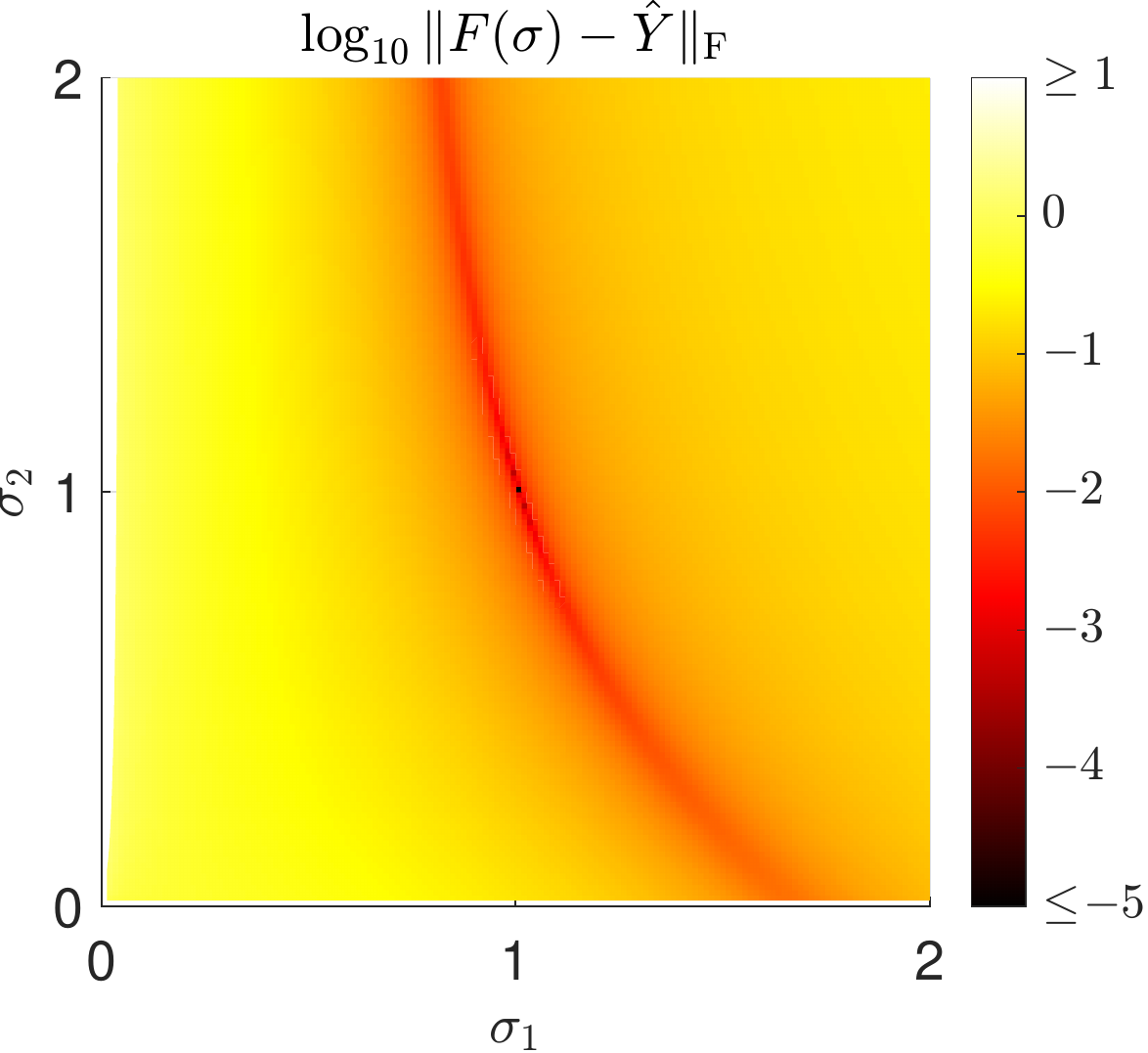}}  &
\mbox{\includegraphics[height=0.25\textheight]{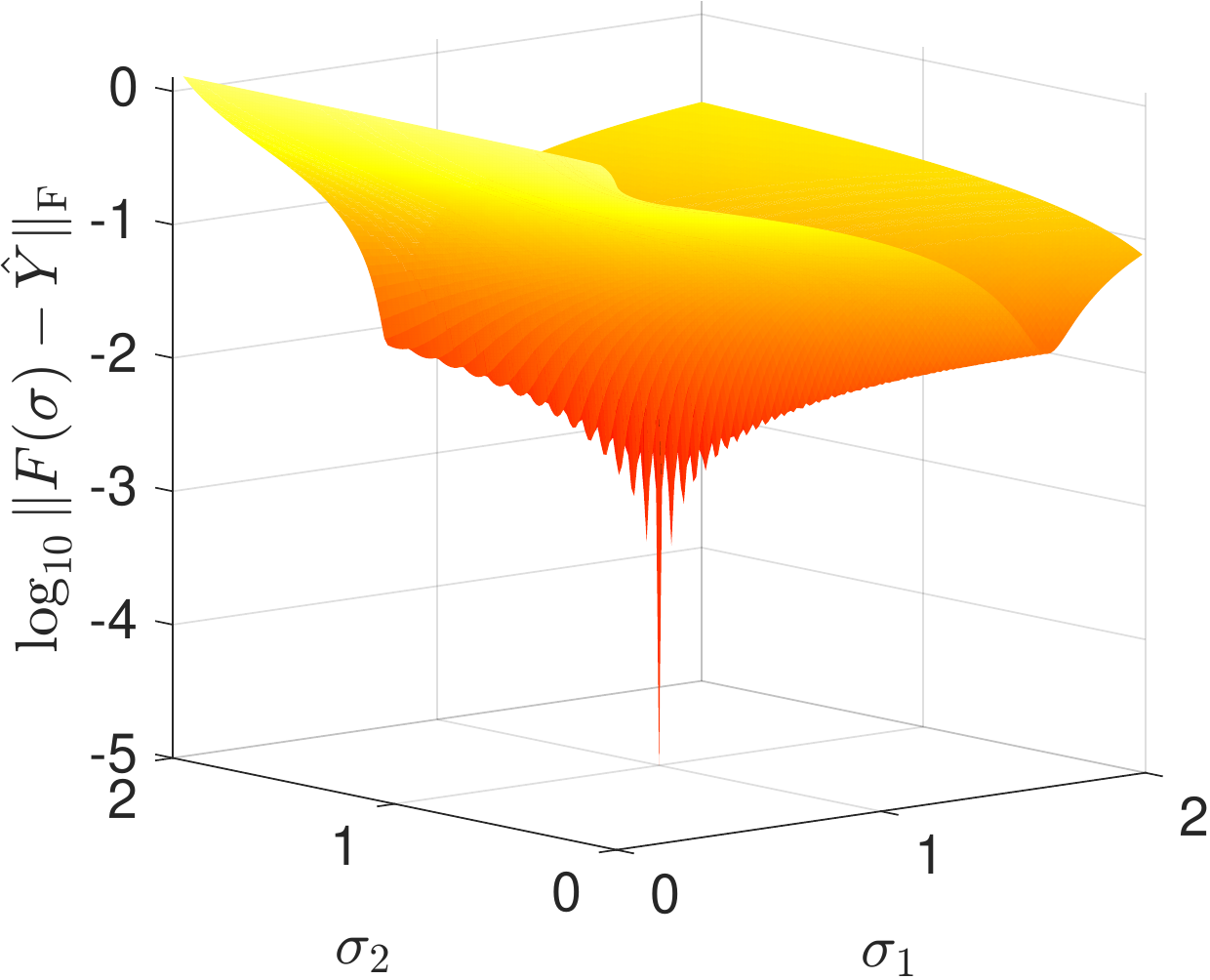}}
\end{tabular}
\end{center}
\caption{Residuum functional for a least-squares data-fitting approach.}
\label{fig:Residual}
\end{figure}

We now demonstrate how standard least-squares data-fitting approaches for this inverse problem may suffer from local minima. We set $\hat\sigma:=(1,1)$ and $\hat Y:=F(\hat \sigma)$ to be the true noiseless measurements of the associated NtD. The natural approach is to find an approximation $\sigma\approx \hat\sigma$ by minimizing the least-squares residuum functional
\begin{equation}\label{eq:LSq_residual}
\norm{F(\sigma)-\hat Y}_F^2\to \text{min!}
\end{equation}
where $\norm{\cdot}_F$ denotes the Frobenius norm. Figure~\ref{fig:Residual} shows the values of this residuum functional as a function of $\sigma=(\sigma_1,\sigma_2)$. It clearly shows a global minimum at the correct value $\sigma=(1,1)$ but also a drastic amount of local minimizers. 

Hence, without a good initial guess, local optimization approaches are prone to end up in local minimizers that may lie far away from the correct coefficient values, and global optimization approaches for such highly non-convex residuum functionals quickly become computationally infeasible for raising numbers of unknowns. 

We demonstrate this problem of local convergence by applying the generic MATLAB solver \verb.lsqnonlin. to the 
least-squares minimization problem \eqref{eq:LSq_residual} (with all options of \verb.lsqnonlin. left on their default values). Figure \ref{fig:Convergence} shows the Euklidian norm of the error  
of the final iterate $\sigma^{(N)}$ returned by \verb.lsqnonlin. as a function of the initial value $\sigma^{(0)}=(\sigma^{(0)}_1,\sigma^{(0)}_2)$. It shows how some regions of starting values lead to inaccurate or even completely wrong results. Note that the values are plotted in logarithmic scale and cropped above and below certain thresholds to improve presentation. Also, in our example, increasing the number of measurements $m$ did not alleviate these problems,
and the performance of \verb.lsqnonlin. did not change when providing the symbolically calculated Jacobian of $F$.

\begin{figure}
\begin{center}
\begin{tabular}{c}
\mbox{\includegraphics[height=0.25\textheight]{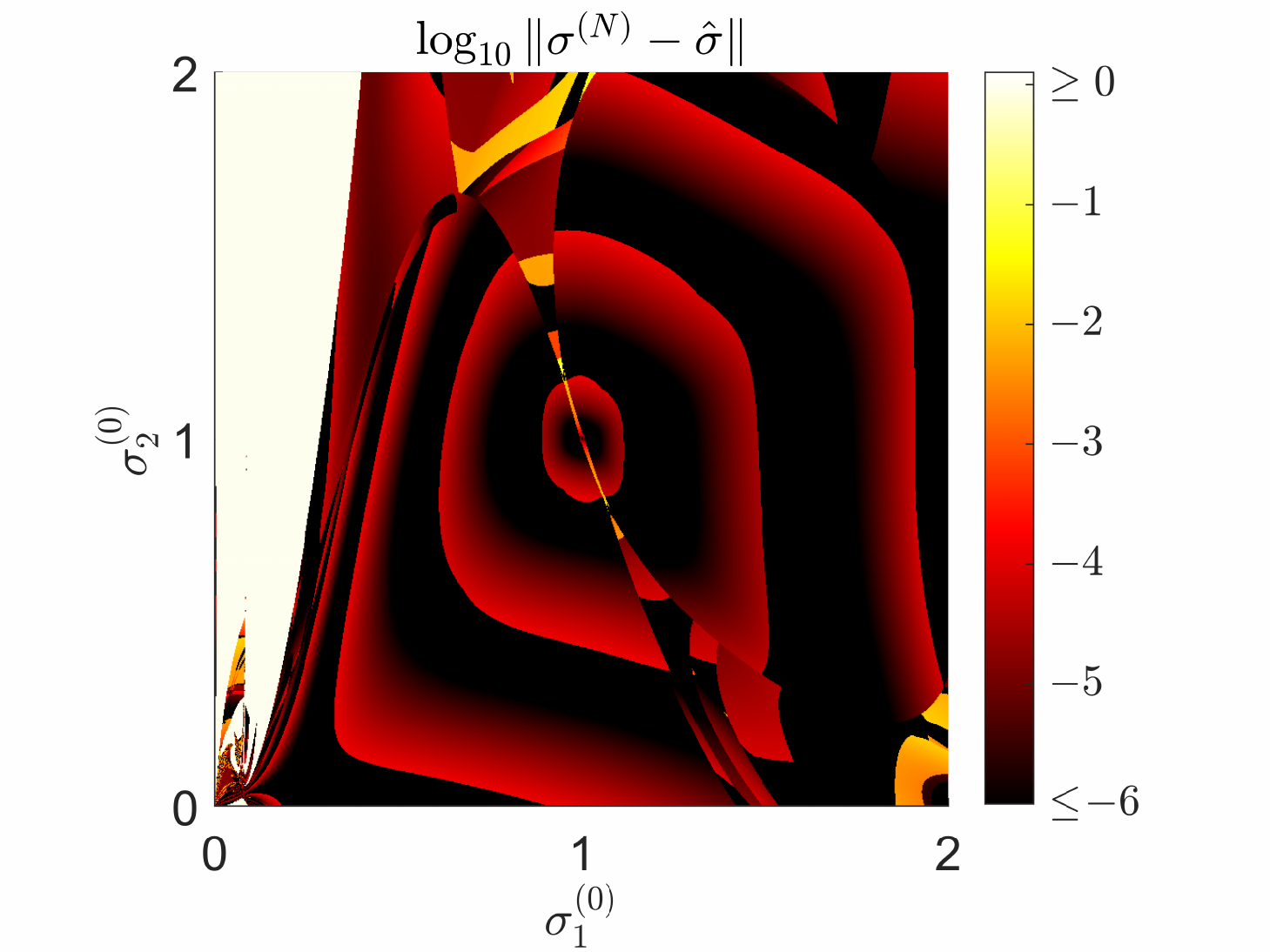}}
\end{tabular}
\end{center}
\caption{Error of the final iterate of a standard minimization algorithm 
when started with different initial values $\sigma^{(0)}=(\sigma^{(0)}_1,\sigma^{(0)}_2)$.}
\label{fig:Convergence}
\end{figure}

Our example illustrates how standard approaches to non-linear inverse coefficient problems (such as the Calder\'on problem) may lead to non-convex residuum functionals that are highly affected by the problem of local minimizers. Overcoming the problem of non-linearity requires overcoming this problem of non-convexity. In the next sections we will show that this is indeed possible. The Calder\'on problem with finitely many unknowns and measurements can be equivalently reformulated as a convex minimization problem.

\pagebreak

\section{Setting and main results}
\subsection{The partial data Calder\'on problem}\label{subsect:setting_Calderon}
Let $\Omega\subset \mathbb{R}^d$, $d\geq 2$ be a Lipschitz bounded domain, 
let $\nu$ denote the outer normal on $\partial \Omega$, and let $\Sigma\subseteq \partial \Omega$ be a relatively open boundary part. $L^\infty_+(\Omega)$ denotes the subset of $L^\infty$-functions with positive essential infima, $\Hd^1(\Omega)$ and $\Ld^2(\Sigma)$ denote the spaces of $H^1$- and $L^2$-functions with vanishing integral mean on $\Sigma$.

For $\sigma\in L^\infty_+(\Omega)$, the partial data (or local) Neumann-Dirichlet-operator
\[
\Lambda(\sigma):\ \Ld^2(\Sigma)\to \Ld^2(\Sigma)
\]
is defined by $\Lambda(\sigma)g:=u\vert_\Sigma$ where $u\in \Hd^1(\Omega)$ solves
\begin{equation}
\label{eq:EIT}
\nabla\cdot (\sigma \nabla u)=0 \quad \text{ in $\Omega$,} \quad \text{ and } \quad
\sigma \partial_\nu u\vert_{\partial \Omega}=\left\{ \begin{array}{l l} 
g & \text{ on $\Sigma$,}\\
0 & \text{ else.}
\end{array}\right.
\end{equation}
Occasionally, we will also write $u_\sigma^g$ for the solution of \eqref{eq:EIT} in the following to highlight the dependence on the conductivity coefficient $\sigma$ and the Neumann boundary data $g$.

It is well-known (and easily follows from standard elliptic PDE theory) that $\Lambda(\sigma)$ is a compact self-adjoint operator from $\Ld^2(\Sigma)$ to $\Ld^2(\Sigma)$.
The question, whether $\Lambda(\sigma)$ uniquely determines $\sigma$, is known as the 
(partial data) \emph{Calder\'on problem}.

\subsection{The Calder\'on problem with finitely many unknowns}

We will now introduce the Calder\'on problem with finitely many unknowns, and formulate our first main result on its equivalent convex reformulation. 

\paragraph{The setting with finitely many unknowns}
We assume that $\Omega$ is decomposed into 
$n\in \N$ pixels, i.e., 
\[
\overline \Omega=\bigcup_{j=1}^n \overline{P_j}.
\]
where $P_1,\ldots,P_n\subseteq \Omega$ are non-empty, pairwise disjoint subdomains with Lipschitz boundaries. We furthermore assume that the pixels are numbered according to their distance from the boundary part $\Sigma$, so that the following holds: For any $j\in \{1,\ldots,n\}$ we define
\[
Q_j:=\bigcup_{i>j} \overline{P_i} 
\]
and assume that, for all $j=1,\ldots,n$, the complement of $Q_j$ in $\overline \Omega$ is connected and contains a non-empty relatively open subset of $\Sigma$. 

We will consider conductivity coefficients $\sigma\in L^\infty_+(\Omega)$ that are piecewise constant with respect to this pixel partition, and assume that we know upper and lower bounds, $b>a>0$. Hence,
\[
\sigma(x)=\sum_{j=1}^n \sigma_j \chi_{P_j}(x) \quad \text{ with } \sigma_1,\ldots,\sigma_n\in [a,b]\subset \mathbb{R}_+,
\]
and $\chi_{P_j}:\ \Omega\to \mathbb{R}$ denoting the characteristic functions on $P_j$.
In the following, with a slight abuse of notation, we identify such a piecewise constant function $\sigma:\ \Omega\to \mathbb{R}$
with its coefficient vector $\sigma=(\sigma_1,\ldots,\sigma_n)^T\in \mathbb{R}^n$.
Accordingly, we consider $\Lambda$ as a non-linear operator
\[
\Lambda:\ \mathbb{R}^n_+\to \LL(L^2_\diamond(\Sigma)),
\]
and consider the problem to
\[
\text{reconstruct } \quad \sigma\in [a,b]^n\subset \mathbb{R}^n_+ \quad \text{ from } \quad \Lambda(\sigma)\in  \LL(L^2_\diamond(\Sigma)).
\]
Here, and in the following, $\mathbb{R}^n_+:=(0,\infty)^n$ denotes the space of all vectors in $\mathbb{R}^n$ containing only positive entries. Also, throughout this work, we write ``$\leq$'' for the componentwise order on $\mathbb{R}^n$. 

\paragraph{Convex reformulation of the Calder\'on problem}
Let ``$\preceq$'' denote the semidefinite (or Loewner) order on the space of self-adjoint operators in $\LL(L^2_\diamond(\Sigma))$, i.e., for all $A=A^*\in \LL(L^2_\diamond(\Sigma))$, and 
$B=B^*\in \LL(L^2_\diamond(\Sigma))$,
\[
A\preceq B \quad \text{ denotes that } \quad \int_{\Sigma} g (B-A)g \dx[s] \geq 0 \quad \text{ for all } g\in L^2_\diamond(\Sigma).
\] 

Note that, for a compact selfadjoint operator $A\in \LL(L^2_\diamond(\Sigma))$,
the eigenvalues can only accumulate in zero by the spectral theorem. Hence, $A$ either possesses a maximal eigenvalue $\lambda_{\max}(A)\geq 0$, or zero is the supremum (though not necessarily the maximum) of the eigenvalues. In the latter case we still write $\lambda_{\max}(A)=0$ for the ease of notation. Thus, for compact selfadjoint $A\in \LL(L^2_\diamond(\Sigma))$, 
\begin{equation}\label{eq:lambda_max_equiv}
A\not\preceq 0 \quad \text{ if and only if } \quad \lambda_{\max}(A)>0. 
\end{equation}


Our first main result is that the Calder\'on problem with infinitely many measurements can be equivalently formulated as a uniquely solvable convex non-linear semidefinite program, and to give an error estimate for noisy measurements. Our arguments also yield unique solvability of the Calder\'on problem and its linearized version in our setting with finitely many unknowns. We include this as part of our theorem for the sake of completeness. It should be stressed though that uniqueness is known to hold for general piecewise analytic coefficients, cf.\ \cite{Koh84,Koh85} for the non-linear Calder\'on problem, and \cite{harrach2010exact} for the linearized version. For more general results on the question of uniqueness we refer to the recent uniqueness results cited in the introduction, and the references therein.

\begin{theorem}\label{thm:main_infinite}
There exists $c\in \mathbb{R}^n_+$, and $\lambda>0$, so that for all $\hat\sigma\in [a,b]^n$, and $\hat Y:=\Lambda(\hat\sigma)$, the following holds:
\begin{enumerate}[(a)]
\item The convex semi-definite optimization problem
\begin{equation}\label{eq:sd_program_nonoise_infinite}
\verb|minimize |\ c^T \sigma\ \verb| subject to | \ \sigma\in [a,b]^n,\ \Lambda(\sigma)\preceq \hat Y,
\end{equation}
possesses a unique minimizer and this minimizer is $\hat \sigma$.
%
\item Given $\delta>0$, and a self-adjoint operator 
\[
Y^\delta\in \LL(L^2_\diamond(\Sigma)), \quad \text{ with } \quad \norm{Y^\delta-\hat Y}_{\LL(L^2_\diamond(\Sigma))}\leq \delta,
\]
the convex semi-definite optimization problem
\begin{equation}\label{eq:sd_program_noise_infinite}
\verb|minimize |\ c^T \sigma\ \verb| subject to | \ \sigma\in [a,b]^n,\ \Lambda(\sigma)\preceq Y^\delta+\delta I,
\end{equation}
possesses a minimizer $\sigma^\delta$. Every such minimizer $\sigma^\delta$ fulfills
\[
 \norm{\sigma^\delta-\hat \sigma}_{c,\infty} \leq \frac{2(n-1)}{\lambda}\delta,
\]
where $\norm{\cdot}_{c,\infty}$ denotes the weighted maximum norm
\[
\norm{\sigma}_{c,\infty}:=\max_{j=1,\ldots,n} c_j \vert \sigma_j \vert \quad \text{ for all } \sigma\in \mathbb{R}^n.
\]
\end{enumerate}
Moreover, the non-linear mapping 
\[
\Lambda:\ \mathbb{R}^n_+\to \LL(L^2_\diamond(\Sigma))
\]
is injective, and its Fr\'echet derivative $\Lambda'(\sigma)\in \LL(\mathbb{R}^n,\LL(L^2_\diamond(\Sigma)))$ is injective for all $\sigma\in \mathbb{R}^n_+$.
\end{theorem}

Let us stress that the constants $c\in \mathbb{R}^n_+$, and $\lambda>0$, in Theorem~\ref{thm:main_infinite} only depend
on the domain $\Omega\subset \mathbb{R}^d$, the pixel partition $P_1,\ldots,P_n$ and the a-priori known bounds $b>a>0$. Moreover, note that \eqref{eq:sd_program_nonoise_infinite}, and \eqref{eq:sd_program_noise_infinite}, are indeed convex optimization problems since 
$\Lambda$ is convex with respect to the Loewner order so that the admissible sets
in \eqref{eq:sd_program_nonoise_infinite}, and \eqref{eq:sd_program_noise_infinite} are closed convex sets (cf.\ Corollary~\ref{cor:inf_meas_set_convex}). The admissible sets in \eqref{eq:sd_program_nonoise_infinite}, and \eqref{eq:sd_program_noise_infinite} are also non-empty since, in both cases, they contain $\hat \sigma$.

\begin{remark}\label{rem:Lipschitz_infinite}
Our arguments also yield the Lipschitz stability result
\[
\norm{\sigma_1-\sigma_2}_{c,\infty} \leq \frac{(n-1)}{\lambda} \norm{\Lambda(\sigma_1)-\Lambda(\sigma_2)}_{\LL(L^2_\diamond(\Sigma))}
\quad \text{ for all } \sigma_1,\sigma_2\in [a,b]^n.
\]
\end{remark}

\subsection{The case of finitely many measurements}
\label{subsect:result_finite}
 The equivalent convex reformulation
is also possible for the case of finitely many measurements. Let $g_1,g_2,\ldots\subseteq  L^2_\diamond(\Sigma)$ be given with dense span in $L^2_\diamond(\Sigma)$.
For some number of measurements $m\in \N$, we assume that we can measure the Galerkin projection 
of $\Lambda(\sigma)$ to the span of $\{g_1,\ldots,g_m\}$, i.e., that we can measure
$\int_{\Sigma} g_j \Lambda(\sigma) g_k \dx[s]$, for all $j,k=1,\ldots,m$. 
Accordingly, we define the matrix-valued forward operator
\begin{equation}\label{eq:def_Fm}
F_m:\ \mathbb{R}^n_+\to \sS_m\subset \mathbb{R}^{m\times m}, \quad F(\sigma):=\left(\int_{\Sigma} g_j \Lambda(\sigma) g_k \dx[s]\right)_{j,k=1,\ldots,m},
\end{equation}
where $\sS_m\subset \mathbb{R}^{m\times m}$ denotes the space of symmetric matrices.

Then, the problem of reconstructing an unknown conductivity on a fixed partition
with known bounds from finitely many measurements can be formulated as the problem to
\[
\text{reconstruct } \quad \sigma\in [a,b]^n\subset \mathbb{R}^n_+ \quad \text{ from } \quad F_m(\sigma)\in \sS_m.
\]
As in the infinite-dimensional case, we write ``$\preceq$'' for the Loewner order on the space of symmetric matrices $\sS_m$, i.e., for all $A,B\in \sS_m$,
\[
A\preceq B \quad \text{ denotes that } \quad x^T(B-A)x\geq 0 \quad \text{ for all } x\in \mathbb{R}^m.
\]
Also, for $A\in \sS_m$, the largest eigenvalue is denoted by $\lambda_{\max}(A)$, and we have that
\begin{equation*}\label{eq:lambda_max_equiv_finite}
A\not\preceq 0 \quad \text{ if and only if } \quad \lambda_{\max}(A)>0. 
\end{equation*}

Our second main result is that also the Calder\'on problem with finitely many measurements can be equivalently reformulated as a uniquely solvable convex non-linear semi-definite program, provided that sufficiently many measurements are being taken. Again, our arguments also yield unique solvability of the Calder\'on problem and its linearized version, and we include this as part of our theorem for the sake of completeness. Note that this uniqueness result has already been shown in 
\cite{harrach2019uniqueness} by arguments closely related to those in this work.

\begin{theorem}\label{thm:main_finite}
If the number of measurements $m\in \N$ is sufficiently large, then: 
\begin{enumerate}[(a)]
\item The non-linear mapping 
\[
F_m:\ [a,b]^n\to \sS_m\subseteq \mathbb{R}^{m\times m}
\]
is injective on $[a,b]^n$, and its Fr\'echet derivative $F'(\sigma)\in \LL(\mathbb{R}^n_+,\sS_m)$ is injective for all $\sigma\in [a,b]^n$.
\item There exists $c\in \mathbb{R}^n_+$, and $\lambda>0$, so that for all $\hat \sigma\in [a,b]^n$, and $\hat Y:=F_m(\hat \sigma)$, the following holds:
\begin{enumerate}[(i)]
\item The convex semi-definite optimization problem
\begin{equation}\label{eq:sd_program_nonoise}
\verb|minimize |\ c^T \sigma\ \verb| subject to | \ \sigma\in [a,b]^n,\ F_m(\sigma)\preceq \hat Y,
\end{equation}
possesses a unique minimizer and this minimizer is $\hat\sigma$.
\item Given $\delta>0$, and $Y^\delta\in \sS_m$ with $\norm{Y^\delta-\hat Y}_{2}\leq \delta$,
the convex semi-definite optimization problem
\begin{equation}\label{eq:sd_program_noise}
\verb|minimize |\ c^T \sigma\ \verb| subject to | \ \sigma\in [a,b]^n,\ F_m(\sigma)\preceq Y^\delta+\delta I,
\end{equation}
possesses a minimizer $\sigma^\delta$. Every such minimizer $\sigma^\delta$ fulfills
\[
 \norm{\sigma^\delta-\hat \sigma}_{c,\infty} \leq \frac{2(n-1)}{\lambda}\delta.
\]
where the weighted maximum norm $\norm{\cdot}_{c,\infty}$ is defined as in Theorem~\ref{thm:main_infinite}.
\end{enumerate}
\end{enumerate}
\end{theorem}

The constants $c\in \mathbb{R}^n_+$, and $\lambda>0$, and also the number of measurements in Theorem~\ref{thm:main_finite} only depend on the domain $\Omega\subset \mathbb{R}^d$, the pixel partition $P_1,\ldots,P_n$ and the a-priori known bounds $b>a>0$. Also, as in the
case of infinitely many measurements, \eqref{eq:sd_program_nonoise} and \eqref{eq:sd_program_noise} are linear optimization problems
over non-empty, closed, and convex, feasibility sets, and our arguments also yield the Lipschitz stability result
\begin{equation}\label{eq:Lipschitz_finite}
\norm{\sigma_1-\sigma_2}_{c,\infty} \leq \frac{(n-1)}{\lambda} \norm{F_m(\sigma_1)-F_m(\sigma_2)}_2
\quad
\forall \sigma_1,\sigma_2\in [a,b]^n.
\end{equation}

\section{Proof of the main results}

In the following, the $j$-th unit vector is denoted by $e_j\in \mathbb{R}^n$.
We write $\1:=(1,1,\ldots,1)^T\in \mathbb{R}^n$ for the vector containing only ones,
and we write $e_j':=\1-e_j$ for the vector containing ones in all entries 
except the $j$-th. We furthermore split $e_j'=e_j^+ + e_j^-$, where 
\[
e_j^+:=\sum_{i=j+1,\ldots,n} e_i, \quad \text{ and } \quad e_j^-:=\sum_{i=1,\ldots,j-1} e_i.
\]
Note that we use the usual convention of empty sums being zero, so that
$e_n^+=0\in \mathbb{R}^n$, and $e_1^{-}=0\in \mathbb{R}^n$. 

\subsection{The case of infinitely many measurements}\label{Subsec:proof_inf_data}

In this subsection, we will prove Theorem~\ref{thm:main_infinite}, where the measurements
are given by the infinite-dimensional Neumann-Dirichlet operator $\Lambda(\sigma)\in \LL(L^2_\diamond(\Sigma))$. 

Let us sketch the main ideas and outline the proof first. 
We will start by summarizing some known results on the monotonicity and convexity of the forward mapping
\[
\Lambda:\ \mathbb{R}^n_+\to \LL(L^2_\diamond(\Sigma)).
\]
This will show that the admissible sets of the optimization problems \eqref{eq:sd_program_nonoise_infinite} and \eqref{eq:sd_program_noise_infinite} are indeed convex sets. It also yields the monotonicity property that, for all $\sigma,\tau\in \mathbb{R}^n_+$,
\[
\tau\geq \sigma \quad \text{ implies } \quad \Lambda(\tau)\preceq \Lambda(\sigma).
\] 
Using localized potentials arguments in the form of directional derivatives, we then 
derive a converse monotonicity result showing that 
\begin{equation}\label{eq:inf_data_converse_mon}
\exists c\in \mathbb{R}^n_+: \quad \Lambda(\tau)\preceq \Lambda(\sigma) \quad \text{ implies } \quad c^T \tau\geq c^T \sigma,
\end{equation} 
for all $\sigma,\tau\in [a,b]^n$. Clearly, this implies that $\hat \sigma$ is a minimizer of the optimization problem \eqref{eq:sd_program_nonoise_infinite}. By proving a slightly stronger variant of \eqref{eq:inf_data_converse_mon}, we also obtain uniqueness of the minimizer and the error estimate in Theorem~\ref{thm:main_infinite}(b).

\paragraph{Monotonicity, convexity, and localized potentials} We collect several known properties of the forward mapping in the following lemma. We also rewrite the localized potentials arguments from \cite{gebauer2008localized,harrach2013monotonicity} as a definiteness property of the directional derivatives of $\Lambda$. The latter will be the basis for proving a converse monotonicity result in the next subsection.

\begin{lemma}\label{lemma:properties_Lambda}
\begin{enumerate}[(a)]
\item $\Lambda$ is Fr\'echet differentiable with continuous derivative 
$$
\Lambda':\ \mathbb{R}^n_+\to \LL(\mathbb{R}^n,\LL(L^2_\diamond(\Sigma))).
$$
$\Lambda'(\sigma)d\in \LL(L^2_\diamond(\Sigma))$ is compact and selfadjoint for all $\sigma\in \mathbb{R}^n_+$, and $d\in \mathbb{R}^n$. Moreover,
\begin{alignat}{2}
\label{eq:monoton_deriv} \Lambda'(\sigma)d&\preceq 0 \quad && \text{for all $\sigma\in \mathbb{R}^n_+$, $0\leq d\in \mathbb{R}^n$,}\\
\label{eq:convex_deriv} \Lambda(\tau)-\Lambda(\sigma)&\succeq \Lambda'(\sigma) (\tau-\sigma) \quad && \text{for all $\sigma,\tau\in \mathbb{R}^n_+$.}
\end{alignat}
\item $\Lambda$ is monotonically decreasing, i.e.
for all $\sigma,\tau\in \mathbb{R}^n_+$
\begin{align}\label{eq:monotonicity}
\sigma\geq \tau \quad \text{ implies } \quad \Lambda(\sigma)\preceq \Lambda(\tau).
\end{align}
Also, for all $\sigma\in \mathbb{R}^n_+$, and $d,\tilde d\in \mathbb{R}^n$,
\begin{align}\label{eq:monotonicity_deriv}
d\geq \tilde d \quad \text{ implies } \quad \Lambda'(\sigma)d\preceq \Lambda'(\sigma)\tilde d.
\end{align}
\item $\Lambda$ is convex, i.e.
for all $\sigma,\tau\in \mathbb{R}^n_+$, and $t\in [0,1]$,
\begin{align}
\Lambda(t \sigma + (1-t) \tau)\preceq t\Lambda(\sigma) + (1-t) \Lambda(\tau).
\end{align}
\item For all $C>0$, $j\in \{1,\ldots,n\}$, and $\sigma\in \mathbb{R}^n_+$, it holds that
\begin{align}
\Lambda'(\sigma)(e_j - C e_j^+)\not\succeq 0.\label{eq:dir_deriv}
\end{align}
\end{enumerate}
\end{lemma}
\begin{proof}
It is well known that $\Lambda$ is continuously Fr\'echet differentiable (cf., e.g., \cite[Section~2]{lechleiter2008newton}, or \cite[Appendix~B]{garde2017convergence}).
Given some direction $d\in \mathbb{R}^n$ the derivative 
\[
\Lambda'(\sigma)d:\ \Ld^2(\partial \Omega)\to \Ld^2(\partial \Omega)
\]
is the selfadjoint compact linear operator associated to the bilinear form
\begin{equation*}\label{eq:Frechet_bilinearForm}
\int_\Sigma g \left(\Lambda'(\sigma)d\right) h \dx[s]
= -\int_\Omega d(x)\, \nabla u^g_\sigma(x)
\cdot \nabla u^h_\sigma(x) \dx,
\end{equation*}
where, again, we identify $d\in \mathbb{R}^n$ with the piecewise constant function 
\[
d(x)=\sum_{j=1}^n d_j \chi_{P_j}(x):\ \Omega\to \R.
\]
$u^g_\sigma$, resp., $u^h_\sigma$ denote the solutions of \eqref{eq:EIT} with Neumann boundary data $g$, resp., $h$. Clearly, this implies \eqref{eq:monoton_deriv}.
Assertion \eqref{eq:convex_deriv} is shown in \cite[Lemma~2.1]{harrach2010exact}, so that (a) is proven. 

The monotonicity results \eqref{eq:monotonicity} and \eqref{eq:monotonicity_deriv} in (b) immediately follow from \eqref{eq:convex_deriv} and \eqref{eq:monoton_deriv}. 

The convexity result in (c) follows from \eqref{eq:convex_deriv} by a standard argument, cf., e.g., \cite[Lemma~2]{harrach2021introduction}.

To prove (d), let $j\in \{1,\ldots,n\}$. Then, for all $g\in \Ld^2(\partial \Omega)$,
\begin{align*}
-\int_\Sigma g \Lambda'(\sigma)(e_j - C e_j^+)g\dx[s]=\int_{P_j} \vert\nabla u^g_\sigma\vert^2\dx - C \int_{Q_j} \vert\nabla u^g_\sigma\vert^2\dx,
\end{align*}
where $Q_j=\bigcup_{i>j} \overline{P_i}$. Since $P_j$ is open and disjoint to $Q_j$, and the complement of $Q_j$ is connected to $\Sigma$, we can apply the localized potentials result in \cite[Thm.~3.6, and Section~4.3]{harrach2013monotonicity} to obtain a sequence of boundary currents $(g_k)_{k\in \N}$ with
\[
\int_{P_j} \vert\nabla u^{g_k}_\sigma\vert^2\dx \to \infty, \quad \text{ and } \quad \int_{Q_j} \vert\nabla u^{g_k}_\sigma\vert^2\dx \to 0.
\]
Hence, for sufficiently large $k\in \N$, 
\begin{align*}
-\int_\Sigma g_k \Lambda'(\sigma)(e_j - C e_j^+)g_k \dx[s]>0,
\end{align*}
which proves (d).
\end{proof}

\begin{corollary}\label{cor:inf_meas_set_convex} For every self-adjoint operator $Y\in \LL(L^2_\diamond(\Sigma))$, the set 
\[
\mathcal C:=\{ \sigma\in [a,b]^n:\ \Lambda(\sigma)\preceq Y\}
\]
is a closed convex set in $\mathbb{R}^n$.
\end{corollary}
\begin{proof}
This follows immediately from the continuity and convexity of $\Lambda$, which was shown in part (a) and (c) of Lemma~\ref{lemma:properties_Lambda}.
%
\end{proof}

\paragraph{A converse monotonicity result}
We will now utilize the properties of the directional derivatives \eqref{eq:dir_deriv} in order to derive the converse monotonicity result \eqref{eq:inf_data_converse_mon} (or, more precisely, a slightly stronger version of it). The following arguments stem on the ideas of \cite{harrach2021uniqueness,harrach2021solving}, where \eqref{eq:inf_data_converse_mon} is proven with $c:=\1$. The main technical difficulty compared to these previous works, is that the localized potentials results for the Calder\'on problem are weaker than those known for the Robin problem considered in \cite{harrach2021uniqueness,harrach2021solving}. This corresponds to the fact that \eqref{eq:dir_deriv} does not contain the term $e_j^-$. We overcome this difficulty by a compactness and scaling approach.

\begin{lemma}\label{lemma:compact_argument}
For all constants $C>0$, and $j\in \{1,\ldots,n\}$, there exists $\delta>0$, so 
that 
\begin{align}\label{eq:dir_deriv_delta}
-\Lambda'(\sigma)(e_j - \delta e_j^- - C e_j^+)\not\preceq 0
\quad \text{ for all } \sigma\in [a,b]^n.
\end{align}
\end{lemma}
\begin{proof}
Let $C>0$, and $j\in \{1,\ldots,n\}$. We define the functions
\begin{alignat*}{2}
\varphi:\ L_\diamond^2(\Sigma)\times \mathbb{R}^n_+ & \to \R,\quad & \varphi(g,\sigma)&:=-\int_\Sigma g \Lambda'(\sigma)(e_j - C e_j^+) g \dx[s],\\
\psi:\ \mathbb{R}^n_+ & \to \R,\quad &
\psi(\sigma)&:=\sup_{g\in L_\diamond^2(\Sigma),\ \norm{g}=1} \varphi(g,\sigma).
\end{alignat*}
Then \eqref{eq:dir_deriv} implies that $\psi(\sigma)>0$ for all $\sigma\in \mathbb{R}^n_+$.

Moreover, $\varphi$ is continuous by lemma \ref{lemma:properties_Lambda}, so that 
$\psi$ is lower semicontinuous. Hence, $\psi$ attains its minimum over the compact set $[a,b]^n$, and it follows that there exists $\epsilon>0$, so that
\[
\psi(\sigma)\geq \epsilon>0 \quad \text{ for all } \sigma\in [a,b]^n.
\]
By continuity and compactness, we also have that
\[
S:=\sup_{\sigma\in [a,b]^n}\norm{\Lambda'(\sigma)e_j^-}_{\LL(L^2_\diamond(\Sigma))}<\infty.
\]
Setting $\delta:=\frac{\epsilon}{2S}>0$, we now obtain, for all $\sigma\in [a,b]^n$,
\begin{align*}
\lefteqn{\sup_{g\in L_\diamond^2(\Sigma),\ \norm{g}=1} \left(-\int_\Sigma g \Lambda'(\sigma)(e_j - \delta e_j^- - C e_j^+) g \dx[s]\right)}\\
&\geq \sup_{g\in L_\diamond^2(\Sigma),\ \norm{g}=1} \varphi(g,\sigma)
- \delta \sup_{g\in L_\diamond^2(\Sigma),\ \norm{g}=1} \left\vert \int_\Sigma g \left(\Lambda'(\sigma)e_j^-\right) g \dx[s]\right\vert\\
&\geq \psi(\sigma) - \delta \norm{\Lambda'(\sigma)e_j^-}_{\LL(L^2_\diamond(\Sigma)}
\geq \epsilon - \delta S=\frac{\epsilon}{2}>0.
\end{align*}
This proves \eqref{eq:dir_deriv_delta}.
\end{proof}

\begin{lemma}\label{lemma:the_deltas}
There exist 
\begin{equation}\label{eq:delta_order}
0< \delta_1\leq \delta_2\leq \dots \leq \delta_{n-1}\leq \delta_n:=1,
\end{equation}
so that for all $\sigma\in [a,b]^n$, and all $j\in \{1,\ldots,n\}$,
\begin{equation}\label{eq:the_deltas_final}
-\Lambda'(\sigma)D(e_j-(n-1)e_j') \not\preceq 0.
\end{equation}
where $D\in \mathbb{R}^{n\times n}$ is the diagonal matrix with diagonal elements $\delta_1,\ldots,\delta_n>0$.
\end{lemma}
\begin{proof}
\begin{enumerate}[(a)]
\item Set $C:=n-1$. We will first prove that there exist $\delta_j$, $j=0,\ldots,n$, 
that fulfill \eqref{eq:delta_order}, and
\begin{equation}\label{eq:the_deltas}
-\Lambda'(\sigma)\left( \delta_j e_j - \delta_{j-1}Ce_j^- - Ce_j^+\right)\not\preceq 0 \quad \forall j\in \{1,\ldots,n\}.
\end{equation}

To prove this, we start with $j=n$ and proceed backwards. By Lemma~\ref{lemma:compact_argument} there exists $0<\delta$ 
with
\[
-\Lambda'(\sigma)\left( e_n - \delta e_n^- - Ce_n^+\right)\not\preceq 0 \quad \text{ for all } \sigma\in [a,b]^n,
\]
and clearly we can choose $\delta \leq C$. Setting $\delta_n:=1$ and $\delta_{n-1}:=\frac{\delta}{C}\leq 1$ this yields that
\[
-\Lambda'(\sigma)\left( \delta_n e_n - \delta_{n-1} C e_n^- - Ce_n^+\right)\not\preceq 0 \quad \text{ for all } \sigma\in [a,b]^n,
\]
so that \eqref{eq:the_deltas} is proven for $j=n$.

Now assume that there exist
$0\leq \delta_{k-1}\leq \dots \leq \delta_n:=1$, so that
\eqref{eq:the_deltas} holds for all $j=k,\ldots,n$ with 
$1< k\leq n$. 
Using again Lemma~\ref{lemma:compact_argument}, we then obtain $0<\delta'\leq C$, so that
\[
-\Lambda'(\sigma)\left( e_{k-1} - \delta' e_{k-1}^- - \frac{C}{\delta_{k-1}} e_{k-1}^+\right)\not\preceq 0
\quad \text{ for all } \sigma\in [a,b]^n.
\]
With $\delta_{k-2}:=\frac{\delta_{k-1} \delta'}{C}\leq \delta_{k-1}$ this yields that
\[
-\Lambda'(\sigma)\left( \delta_{k-1} e_{k-1} - \delta_{k-2} C e_{k-1}^- - C e_{k-1}^+\right)\not\preceq 0,
\]
which proves \eqref{eq:the_deltas} for $j=k-1$. Hence, by induction, \eqref{eq:the_deltas} can be fulfilled for all $j\in \{1,\ldots,n\}$.
\item To see that \eqref{eq:the_deltas} also implies \eqref{eq:the_deltas_final}, note that
for $j\in \{1,\ldots,n\}$
\[
De_j=\delta_j e_j, \quad De_j^-\leq \delta_{j-1} e_j^-, \quad \text{ and } \quad
De_j^+\leq e_j^+.
\]
Hence, for all $\sigma\in [a,b]^n$, we obtain from \eqref{eq:the_deltas} and the monotonicity property \eqref{eq:monotonicity_deriv} that
\begin{align*}
-\Lambda'(\sigma)D(e_j-(n-1)e_j')
&\succeq 
-\Lambda'(\sigma)\left( \delta_j e_j - \delta_{j-1} (n-1) e_j^- - (n-1)e_j^+\right) 
\not\preceq 0,
\end{align*}
so that \eqref{eq:the_deltas_final} is proven for $j\in \{1,\ldots,n\}$.
\end{enumerate}
\end{proof}

\begin{lemma}\label{lemma:inverse_monotonicity_Lambda}
Let $D\in \mathbb{R}^{n\times n}$ be a diagonal matrix with entries $\delta_1,\ldots,\delta_n>0$,
and assume that, for some $\sigma\in \mathbb{R}^n_+$,
\begin{equation}\label{eq:sum_crit_general}
\Lambda'(\sigma)D((n-1)e_j'-e_j)\not\preceq 0 \quad \text{ for all } \quad  j\in \{1,\ldots,n\}.
\end{equation}
Then
\begin{equation}\label{eqref:def_lambda_sigma}
\lambda:=\min_{j=1,\ldots,n} \lambda_{\max}(\Lambda'(\sigma)D((n-1)e_j'-e_j))>0,
\end{equation}
and, for all $d\in \mathbb{R}^n$,
\begin{equation}\label{eq:conv_mon_sum_aux}
\lambda_{\max}(\Lambda'(\sigma)d)< \frac{\lambda \norm{D^{-1}d}_\infty}{n-1}
 \quad \text{ implies } \quad
\min_{j=1,\ldots,n} \frac{d_j}{\delta_j} > -\frac{1}{n-1} \max_{j=1,\ldots,n} \frac{d_j}{\delta_j}.
\end{equation}
\end{lemma}
\begin{proof}
Clearly, \eqref{eq:sum_crit_general} implies $\lambda>0$. 
We now prove \eqref{eq:conv_mon_sum_aux} by contraposition and assume that there exists an index $k\in \{1,\ldots,n\}$ with 
\[
\frac{d_k}{\delta_k}=\min_{j=1,\ldots,n} \frac{d_j}{\delta_j}\leq -\frac{1}{n-1} \max_{j=1,\ldots,n} \frac{d_j}{\delta_j}.
\]
We have that either 
\[
\norm{D^{-1} d}_\infty=\max_{j=1,\ldots,n} \frac{d_j}{\delta_j},\quad \text{ or } \quad
\norm{D^{-1} d}_\infty=-\min_{j=1,\ldots,n} \frac{d_j}{\delta_j}=-\frac{d_k}{\delta_k},
\]
and in both cases it follows that
\[
\frac{d_k}{\delta_k}\leq -\frac{1}{n-1}\norm{D^{-1} d}_\infty.
\]
This yields
\[
D^{-1} d\leq -\frac{1}{n-1}\norm{D^{-1} d}_\infty e_k + \norm{D^{-1}d}_\infty e_k'
=\frac{\norm{D^{-1} d}_\infty}{n-1} \left( (n-1)e_k' - e_k \right),
\]
and thus
\[
d\leq \frac{\norm{D^{-1} d}_\infty}{n-1} D\left( (n-1)e_k' - e_k \right).
\]

Hence, by monotonicity,
\[
\Lambda'(\sigma) d \succeq \frac{\norm{D^{-1} d}_\infty}{n-1} \Lambda'(\sigma)  D\left( (n-1)e_k' - e_k \right)
\]
and we then obtain from \eqref{eqref:def_lambda_sigma} that
\[
\lambda_{\max}(\Lambda'(\sigma) d)\geq \frac{\norm{D^{-1} d}_\infty}{n-1} \lambda,
\]
so that \eqref{eq:conv_mon_sum_aux} is proven.
\end{proof}

From the preceding lemmas, we can now deduce our converse montonicity result.
\begin{corollary}\label{cor:inf_meas_converse_monotonicity}
There exist $c\in\mathbb{R}^n_+$, and $\lambda>0$, so that the following holds.
\begin{enumerate}[(a)]
\item For all $\sigma\in [a,b]^n$, and $d\in \mathbb{R}^n$, 
\[
\lambda_{\max}(\Lambda'(\sigma)d)< \frac{\lambda \norm{d}_{c,\infty}}{n-1} \quad  \text{ implies } \quad c^T d> 0,
\]
and thus, a fortiori, for $d\neq 0$,
\[
\Lambda'(\sigma)d\preceq 0 \quad  \text{ implies } \quad c^T d> 0.
\]
\item For all $\sigma,\tau \in [a,b]^n$,
\[
\lambda_{\max}(\Lambda(\tau)-\Lambda(\sigma))<  \frac{\lambda \norm{\tau-\sigma}_{c,\infty}}{n-1}   \quad \text{ implies } \quad c^T \tau> c^T\sigma,
\]
and thus, a fortiori, for $\sigma\neq \tau$,
\[
\Lambda(\tau)\preceq \Lambda(\sigma) \quad \text{ implies } \quad c^T \tau > c^T \sigma.
\]
\item The non-linear forward mapping $\Lambda:\ \mathbb{R}^n_+\to \LL(\Ld^2(\Sigma))$ is injective.
\item For all $\sigma\in \mathbb{R}^n_+$, the linearized forward mapping $\Lambda'(\sigma)\in \LL(\mathbb{R}^n,\LL(\Ld^2(\Sigma)))$ is injective.
\end{enumerate}
\end{corollary}
\begin{proof}
Lemma~\ref{lemma:the_deltas} yields a diagonal matrix 
$D\in \mathbb{R}^{n\times n}$ with diagonal elements $\delta_1,\ldots,\delta_n>0$,
so that 
\begin{equation*}
\lambda_{\max}(\Lambda'(\sigma)D((n-1)e_j'-e_j))>0 \quad \text{ for all } \sigma\in [a,b]^n,\ j=1,\ldots,n.
\end{equation*}
Since
\[
\sigma \mapsto \min_{j=1,\ldots,n}\lambda_{\max}(\Lambda'(\sigma)D((n-1)e_j'-e_j))
\]
is a continuous mapping from $[a,b]^n\to \mathbb{R}$, we obtain by compactness that
\[
\exists \lambda>0: \quad \min_{j=1,\ldots,n}\lambda_{\max}(\Lambda'(\sigma)D((n-1)e_j'-e_j))\geq \lambda  \quad \text{ for all } \sigma\in [a,b]^n.
\] 
Setting $c^T:=\begin{pmatrix} \frac{1}{\delta_1} & \dots & \frac{1}{\delta_n}\end{pmatrix}$, (a) follows from Lemma~\ref{lemma:inverse_monotonicity_Lambda}.

(b) follows from (a) as the convexity property \eqref{eq:convex_deriv} yields that
\[
\lambda_{\max}(\Lambda'(\sigma)(\tau-\sigma))\leq \lambda_{\max}(\Lambda(\tau)-\Lambda(\sigma)) \quad \text{ for all } \sigma,\tau\in [a,b]^n.
\]

Clearly, (b) implies injectivity of $\Lambda:\ [a,b]^n\to \LL(\Ld^2(\Sigma))$.
Since this holds for arbitrary large intervals $[a,b]^n$, we obtain injectivity on all of $\mathbb{R}^n_+$, so that (c) is proven. Likewise, (d) follows from (a).
\end{proof}

\paragraph{Proof of Theorem~\ref{thm:main_infinite} and Remark~\ref{rem:Lipschitz_infinite}}

We can now prove Theorem~\ref{thm:main_infinite} with $c\in \mathbb{R}^n_+$, and $\lambda>0$ as given by Corollary~\ref{cor:inf_meas_converse_monotonicity}.
First note
that Corollary~\ref{cor:inf_meas_set_convex} ensures that \eqref{eq:sd_program_nonoise_infinite} and \eqref{eq:sd_program_noise_infinite} are linear optimization problems on convex sets. 
 
Let $\hat\sigma\in [a,b]^n$, and $\hat Y:=\Lambda(\hat \sigma)$. Then $\hat\sigma$ is feasible for the optimization problem \eqref{eq:sd_program_nonoise_infinite}.
For every other feasible $\sigma\in [a,b]^n$, $\sigma\neq \hat\sigma$, the feasibility yields that $\Lambda(\sigma)\preceq \hat Y=\Lambda(\hat\sigma)$, so that $c^T \sigma>c^T\hat \sigma$ by Corollary~\ref{cor:inf_meas_converse_monotonicity}(b). Hence, $\hat\sigma$ is the unique minimizer of 
the convex semi-definite optimization problem \eqref{eq:sd_program_nonoise_infinite}, and thus (a) is proven.

To prove (b), let $\delta>0$, and $Y^\delta\in \LL(L^2_\diamond(\Sigma))$ be self-adjoint with
$\norm{Y^\delta-\hat Y}_{\LL(L^2_\diamond(\Sigma))}\leq \delta$. Since the constraint set of 
\eqref{eq:sd_program_noise_infinite} is non-empty and compact, and the cost function is continuous, there exists a minimizer $\sigma^\delta$. Since also $\hat\sigma$ is feasible for \eqref{eq:sd_program_noise_infinite}, it follows that 
$c^T \sigma^\delta\leq c^T \hat\sigma$. By contraposition of Corollary~\ref{cor:inf_meas_converse_monotonicity}(b) we obtain that
\begin{align*}
\frac{\lambda\norm{\sigma^\delta-\hat\sigma}_{c,\delta}}{n-1}\leq \lambda_{\max}(\Lambda(\sigma^\delta)-\Lambda(\hat \sigma))
\leq \lambda_{\max}(Y^\delta + \delta I - \hat Y)\leq 2\delta.
\end{align*}
Hence, (b) follows, and the same argument also yields the Lipschitz stability result in Remark~\ref{rem:Lipschitz_infinite}.
The injectivity results are proven in Corollary~\ref{cor:inf_meas_converse_monotonicity},(c) and (d).
\hfill $\Box$

\subsection{The case of finitely many measurements
}\label{Subsec:proof_finite_data}

We will now treat the case of finitely many measurements. As introduced in subsection~\ref{subsect:result_finite}, let $g_1,g_2,\ldots\in  L^2_\diamond(\Sigma)$ have dense span in $L^2_\diamond(\Sigma)$, and 
consider the finite-dimensional forward operator
\begin{equation*}
F_m:\ \mathbb{R}^n_+\to \sS_m\subset \mathbb{R}^{m\times m}, \quad F(\sigma):=\left(\int_{\Sigma} g_j \Lambda(\sigma) g_k \dx[s]\right)_{j,k=1,\ldots,m},
\end{equation*}
with $m\in \N$.

\paragraph{Monotonicity, convexity, and localized potentials} 

Again, we start by summarizing the monotonicity, convexity, and localized potentials properties of the forward operator.

\begin{lemma}\label{lemma:properties_Fm}
\begin{enumerate}[(a)]
\item $F_m$ is Fr\'echet differentiable with continuous derivative 
\[
F_m':\ \mathbb{R}^n_+\to \LL(\mathbb{R}^n,\mathbb{R}^{m\times m}).
\]
$F_m'(\sigma)d\in \mathbb{R}^{m\times m}$ is symmetric for all $\sigma\in \mathbb{R}^n_+$, and $d\in \mathbb{R}^n$. Moreover,
\begin{alignat}{2}
\label{eq:monoton_deriv_Fm} F_m'(\sigma)d&\preceq 0 \quad && \text{for all $\sigma\in \mathbb{R}^n_+$, $0\leq d\in \mathbb{R}^n$,}\\
\label{eq:convex_deriv_Fm} F_m(\tau)-F_m(\sigma)&\succeq F_m'(\sigma) (\tau-\sigma) \quad && \text{for all $\sigma,\tau\in \mathbb{R}^n_+$.}
\end{alignat}
\item $F_m$ is monotonically decreasing, i.e.
for all $\sigma,\tau\in \mathbb{R}^n_+$
\begin{align}\label{eq:monotonicity_Fm}
\sigma\geq \tau \quad \text{ implies } \quad F_m(\sigma)\preceq F_m(\tau),
\end{align}
Also, for all $\sigma\in \mathbb{R}^n_+$, and $d,\tilde d\in \mathbb{R}^n$,
\begin{align}\label{eq:monotonicity_deriv_Fm}
d\geq \tilde d \quad \text{ implies } \quad F_m'(\sigma)d\preceq F_m'(\sigma)\tilde d.
\end{align}
\item $F_m$ is convex, i.e.
for all $\sigma,\tau\in \mathbb{R}^n_+$, and $t\in [0,1]$,
\begin{align}
F_m(t \sigma + (1-t) \tau)\preceq t F_m(\sigma) + (1-t) F_m(\tau).
\end{align}
\item For all $C>0$, there exists $M\in \N$, so that
\begin{align}
\label{eq:dir_deriv_Fm}
F_m'(\sigma)(e_j - C e_j^+)\not\succeq 0 \quad \text{ for all } \sigma\in [a,b]^n,\ j\in \{1,\ldots,n\},\ m\geq M.
\end{align}
\end{enumerate}
\end{lemma}
\begin{proof}
For all $\sigma\in [a,b]^n$, $m\in \N$, and $v\in \mathbb{R}^m$, we have that
\[
v^T F_m(\sigma) v=\int_{\Sigma} g \Lambda(\sigma) g\dx[s],
\quad \text{ with } g=\sum_{j=1}^m v_j g_j.
\]
Hence, the monotonicity and convexity properties (a), (b), and (c), immediately 
carry over from that of the infinite-dimensional forward operator $\Lambda$
in Lemma~\ref{lemma:properties_Lambda}.

To prove (d), it clearly suffices to show that, for all $C>0$,
\begin{align*}
\exists m\in \N:\quad F_m'(\sigma)(e_j - C e_j^+)\not\succeq 0\quad \text{ for all } \sigma\in [a,b]^n,\ j\in \{1,\ldots,n\}.
\end{align*}
We argue by contradiction, and assume that there exists $C>0$, so that 
\begin{align}\label{eq:loc_pot_finite_aux}
\forall m\in \N:\ F_m'(\sigma_m)(e_{j_m} - C e_{j_m}^+)\succeq 0
\ \text{ for some } \sigma_m\in [a,b]^n,\ j_m\in \{1,\ldots,n\}.
\end{align}
By compactness, after passing to a subsequence, we can assume that $\sigma_m\to \sigma$, and $j_m=j$ for some $\sigma\in [a,b]^n$, and $j\in \{1,\ldots,n\}$. 
Since \eqref{eq:loc_pot_finite_aux} also implies that
\[
F_k'(\sigma_m)(e_{j_m} - C e_{j_m}^+)\succeq 0
\quad \text{ for all } k\in \N,\ m\geq k,
\]
it follows by continuity that
\[
F_k'(\sigma)(e_{j} - C e_{j}^+)\succeq 0
\quad \text{ for all } k\in \N.
\]
However, since $g_1,g_2,\ldots\in  L^2_\diamond(\Sigma)$ have dense span in $L^2_\diamond(\Sigma)$, this would imply that
\[
\Lambda'(\sigma)(e_{j} - C e_{j}^+)\succeq 0,
\]
and thus contradict Lemma~\ref{lemma:properties_Lambda}.
\end{proof}

As in the infinite-dimensional case, the continuity and convexity properties of the forward mapping yield that the admissible set of the optimization problems \eqref{eq:sd_program_nonoise} and \eqref{eq:sd_program_noise} is closed and convex.

\begin{corollary}\label{cor:fin_meas_set_convex} For all $m\in \N$, and every symmetric matrix $Y\in \mathbb{R}^{m\times m}$, the set 
\[
\mathcal C:=\{ \sigma\in [a,b]^n:\ F_m(\sigma)\preceq Y\}
\]
is a closed convex set in $\mathbb{R}^n$.
\end{corollary}
\begin{proof}
This follows from Lemma~\ref{lemma:properties_Fm}.
\end{proof}

\paragraph{A converse monotonicity result}
We will now show that the converse monotonicity 
result in Corollary~\ref{cor:inf_meas_converse_monotonicity} still holds for the case of finitely many (but sufficiently many) measurements.

\begin{lemma}\label{lemma:fin_meas_converse_monotonicity}
There exist $c\in\mathbb{R}^n_+$, and $\lambda>0$, so that for sufficiently large $m\in \N$, the following holds.
\begin{enumerate}[(a)]
\item For all $\sigma\in [a,b]^n$, and $d\in \mathbb{R}^n$, 
\[
\lambda_{\max}(F_m'(\sigma)d)< \frac{\lambda \norm{d}_{c,\infty}}{n-1} \quad  \text{ implies } \quad c^T d> 0,
\]
and thus, a fortiori, for $d\neq 0$,
\[
F_m'(\sigma)d\preceq 0 \quad  \text{ implies } \quad c^T d> 0.
\]
\item For all $\sigma,\tau \in [a,b]^n$,
\[
\lambda_{\max}(F_m(\tau)-F_m(\sigma))<  \frac{\lambda \norm{\tau-\sigma}_{c,\infty}}{n-1}   \quad \text{ implies } \quad c^T \tau> c^T\sigma,
\]
and thus, a fortiori, for $\sigma\neq \tau$,
\[
F_m(\tau)\preceq F_m(\sigma) \quad \text{ implies } \quad c^T \tau > c^T \sigma.
\]
\item The non-linear forward mapping $F_m:\ \mathbb{R}^n_+\to \LL(\Ld^2(\Sigma))$ is injective.
\item For all $\sigma\in \mathbb{R}^n_+$, the linearized forward mapping $F_m'(\sigma)\in \LL(\mathbb{R}^n,\LL(\Ld^2(\Sigma)))$ is injective.
\end{enumerate}
\end{lemma}
\begin{proof}
Using Lemma~\ref{lemma:the_deltas}, we obtain
a diagonal matrix $D\in \mathbb{R}^{n\times n}$ with diagonal elements 
 $\delta_1,\ldots,\delta_n>0$, so that 
for all $\sigma\in [a,b]^n$
\begin{equation*}
-\Lambda'(\sigma)D(e_j-(n-1)e_j') \not\preceq 0.
\end{equation*}
By the same compactness argument as in the proof of Lemma~\ref{lemma:properties_Fm}(d), it follows that there exists $M\in \N$ with
\begin{equation*}
-F_m'(\sigma)D(e_j-(n-1)e_j') \not\preceq 0 \quad \text{ for all } m\geq M.
\end{equation*}
Using compactness again we get that
\[
\lambda:=\min_{\sigma\in [a,b]^n} \lambda_{\max}\left( -F_M'(\sigma)D(e_j-(n-1)e_j')\right)>0,
\]
and thus 
\[
\lambda_{\max} \left( -F_m'(\sigma)D(e_j-(n-1)e_j')\right)\geq \lambda \quad \text{ for all } \sigma\in [a,b]^n,\ m\geq M.
\]
Setting $c^T:=\begin{pmatrix} \frac{1}{\delta_1} & \cdots & \frac{1}{\delta_n}\end{pmatrix}$, and applying Lemma~\ref{lemma:inverse_monotonicity_Lambda}
with $F_m$, $m\geq M$, in place of $\Lambda$, assertion (a) follows. Assertion (b)--(d) 
follow as in Corollary~\ref{cor:inf_meas_converse_monotonicity}.
\end{proof}

\paragraph{Proof of Theorem~\ref{thm:main_finite}}
Theorem~\ref{thm:main_finite} and the Lipschitz stability result \eqref{eq:Lipschitz_finite} now follow from Lemma~\ref{lemma:fin_meas_converse_monotonicity} exactly as in the infinite-dimensional case. \hfill $\Box$

\section{Conclusions and Outlook}

We conclude this work with some remarks on the applicability of our results and possible extensions. The Calder\'on problem is infamous for its high degree of non-linearity and ill-posedness. The central point of this work is to show that the high non-linearity of the problem does not inevitably lead to the problem of local convergence (resp., local minima) demonstrated in section~\ref{section:Motivation}, but that convex reformulations are possible. In that sense, our result proves that it is principally possible to overcome the problem of non-linearity in the Calder\'on problem with finitely many unknowns.

Let us stress that this is purely a theoretical existence result and that our proofs are non-constructive in three important aspects: We show that there exists a number of measurements that uniquely determine the unknown conductivity values and allow for the convex reformulation, but we do not have a constructive method for calculating the required number of measurements. We show that the problem is equivalent to minimizing a linear functional under a convex constraint, but we do not have a constructive method for calculating the vector $c\in \R^n$ defining the linear functional. Also, we derive an error bound for the case of noisy measurements, but we do not have a constructive method for determining the error bound constant $\lambda>0$.

For the simpler, but closely related, non-linear inverse problem of identifying a Robin coefficient \cite{harrach2019global,harrach2021uniqueness,harrach2021solving}, constructive answers to these three issues are known. \cite{harrach2021solving} gives an explicit (and easy-to-check) criterium to identify whether the number of measurements is sufficiently high for uniqueness and convex reformulation, and to explicitly calculate the stability constant $\lambda$ in the error bound. The criterion is based on 
checking a definiteness property of certain directional derivatives in only finitely many evaluations points.
A similar approach might be possible for the Calder\'on problem considered in this work, but the directional derivative arguments are technically much more involved and the extension of the arguments in \cite{harrach2021solving} is far from trivial.
Moreover, for the Robin problem one can simply choose $c:=\1$, but this is based on stronger localized potentials results that do not hold for the Calder\'on problem. Finding methods to constructively characterize $c$ for the Calder\'on problem will be an important topic for further research. At that point, let us however stress again that our result shows that $c$ only depends on the given setting (i.e., the domain and pixel partition and the upper and lower conductivity bounds), but not on the unknown solution. Hence, for a fixed given setting, one could try to determine $c$ in an offline phase, e.g., by calculating $F_m(\sigma)$ for several samples of $\sigma$, and adapting $c$ to these samples.

\bibliography{literaturliste}
\bibliographystyle{abbrv}

\end{document}